\documentclass[12pt]{article}
\usepackage[utf8]{inputenc}
\usepackage[letterpaper, total={6.5in, 9in}]{geometry}
\usepackage{amsmath, amsthm, amssymb}
\usepackage{authblk}

\usepackage{tikz}
\usepackage{hyperref}
\hypersetup{
    colorlinks = true,
    citecolor = red,
    linkcolor = cyan
}

\usepackage{tikz-qtree}
\usepackage{float}

\theoremstyle{plain}
\newtheorem{theorem}{Theorem}[section]
\newtheorem{corollary}[theorem]{Corollary}
\newtheorem{lemma}[theorem]{Lemma}
\newtheorem{proposition}[theorem]{Proposition}

\theoremstyle{definition}
\newtheorem{definition}[theorem]{Definition}

\theoremstyle{remark}
\newtheorem{remark}[theorem]{Remark}

\newcommand{\ZZ}{\mathbb{Z}}

\def\modd#1 #2{#1\ \mbox{\rm (mod}\ #2\mbox{\rm )}}

\title{The $p$-Adic Valuation Trees for Quadratic Polynomials for Odd Primes}
\author{Will Boultinghouse \\
Kentucky Wesleyan College \\
Division of Natural Sciences and Mathematics \\
3000 Frederica Street \\
Owensboro, KY 42301 \\
USA
\and
Emily Hammett \\
Rowan University \\
Department of Mathematics \\
201 Mullica Hill Road \\
Glassboro, NJ 08028 \\
USA
\and
Stephen Hu \\
Rutgers University \\
Department of Mathematics \\
Hill Center for the Mathematical Sciences\\ 
110 Frelinghuysen Road \\
Piscataway, NJ 08854 \\ 
USA
\and
Olena Kozhushkina\footnote{Corresponding author.} \\
Ursinus College \\
Department of Mathematics and Computer Science \\
Pfahler Hall 101 \\
Collegeville, PA 19426 \\
USA\\
\href{mailto:okozhushkina@ursinus.edu}{\tt okozhushkina@ursinus.edu}\\
\and
Rachel Snyder \\
Western Washington University \\
Department of Mathematics \\
516 High Street \\
Bellingham, WA 98225 \\
USA
\and
Justin Trulen \\
Kentucky Wesleyan College \\
Division of Natural Sciences and Mathematics \\
3000 Frederica Street \\
Owensboro, KY 42301 \\
USA
}
\date{}

\begin{document}

\maketitle

\begin{abstract}
    We examine the behavior of the sequences of $p$-adic valuations of quadratic polynomials with integer coefficients for an odd prime $p$ through tree representations. 
    Under this representation, a finite tree corresponds to a periodic sequence, and an infinite tree corresponds to an unbounded sequence.
    We use the polynomial coefficients to determine whether the $p$-adic valuation trees are finite or infinite, the number of infinite branches, the number of levels, the valuations at terminating nodes, and their relationship to the corresponding sequences.
\end{abstract}


\section{Introduction}
Let $p$ be an odd prime. The $p$-adic valuation of a natural number $n$ is defined as the highest power of $p$ that divides $n$. If we write $n=p^{k}d$, where $k\in\mathbb{N}=\left\{0,1,2,\ldots\right\}$ and $d$ is an integer not divisible by $p$, then the $p$-adic valuation is defined as $\nu_{p}(n):=k$. A natural question to investigate is to classify the behavior of the sequence $(\nu_p(x_n))$, given some prime $p$ and an integer sequence $(x_n)$. For instance, Legendre's formula \cite{Legendre} gives a closed form for the sequence $(\nu_p(n!))$:
\begin{equation*}
\nu_{p}(n!) = \frac{n - s_{p}(n)}{p-1},
\end{equation*}
\noindent
where $s_{p}(n)$ denotes the sum of the base-$p$ digits of $n$. 
Medina et al.\ \cite{Medina} considered sequences of $p$-adic valuations generated by polynomials; Bell \cite{Bell} investigated sequences generated by polynomials with rational coefficients to study $p$-regular sequences; Byrnes et al. \cite{Byrnes} considered sequences of $p$-adic valuations for polynomials of the form $f(n)=an^2+c$; Boultinghouse et al. \cite{Boultinghouse} examined the sequences of 2-adic valuations generated by quadratic polynomials with integer coefficients. In this paper, we study the $p$-adic valuation of sequences generated by quadratic polynomials of the form $f(n)=an^2+bn+c$, where $a,b,c\in\mathbb{Z}$ and $a\neq0$, for an \textit{odd} prime $p$.
If $p\mid\gcd(a,b,c)$, then define $k=\min\{\nu_{p}(a),\nu_{p}(b),\nu_{p}(c)\}$. It suffices to consider the corresponding function $f(n)p^{-k}$ since the valuations for $f(n)$ are those for $f(n)p^{-k}$ increased by the value of $k$. Therefore, unless otherwise specified, we assume $p\nmid\gcd(a,b,c)$ throughout.

We construct tree diagrams for each sequence to visualize these sequences. The top node of the tree represents the valuation of $f(n)$ for any $n\in \mathbb{N}$. If $\nu_p(f(n))$ does not depend on $n$, we stop the construction and label the node with its $p$-adic valuation: the sequence $(\nu_p(f(n)))$ is constant. Otherwise, the top node splits into $p$ branches, each branch corresponding to a subsequence $pq,pq+1,pq+2,\ldots,pq+p-1$ for some $q \in \mathbb{N}$. We then repeat the evaluation process for each node to construct the tree, splitting when necessary. See Figure\ \ref{Example:FirstFiniteTree} for example.

\begin{figure}[H]
 \begin{center}
        \begin{tikzpicture}[level distance=50pt, sibling distance=8pt]
		\Tree [.$q$ 
			\edge node[auto=right]{$3q$}; 
			[.$\nu_3\geq 1$
				\edge node[auto=right,yshift=-3]{$9q$}; $\boxed{\nu_3=3}$
				\edge node[auto=right,rotate=90,yshift=5,xshift=10,scale=0.75]{$9q+3$}; $\boxed{\nu_3=2}$
				\edge node[auto=left,yshift=-3]{$9q+6$}; $\boxed{\nu_3=2}$
			]
			\edge node[auto=right,xshift=3]{$3q+1$}; $\boxed{\nu_3=0}$
			\edge node[auto=left]{$3q+2$}; $\boxed{\nu_3=0}$
		]
	    \end{tikzpicture}
    \end{center} 
    \caption{a $3$-adic valuation tree of $f(n)=n^2+27$} 
    \label{Example:FirstFiniteTree}
\end{figure}
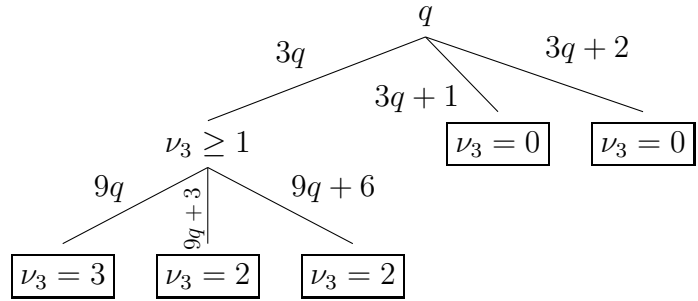

We call any node with no variation \textit{terminating} since it does not split. Otherwise, we call the node \textit{non-terminating}. 

As mentioned above, we represent the terms of the subsequence as nodes. For each $m \in \mathbb{N}$, if $n = p^m q + r$ is a non-terminating node, then the successive branches are $p^m (pq) + r, p^m (pq+1) + r, \ldots, p^m (pq + p-1) + r$, and they are arranged from left to right in the order presented.

Define the \textit{$m^{th}$ level of a tree} as the nodes that can be written in the form $n=p^{m}q+r$, where $r\in\left\{0,1,2,\ldots,p^{m}-1\right\}$. We call a tree \textit{finite} if there exists a level that contains only terminating nodes. Otherwise, we refer to such a tree as \textit{infinite}. In the case of a finite tree, we define the \textit{number of levels} as the smallest $m$ such that the valuation of $f(p^{m}q+r)$ is constant for all $r$ with respect to $q\in\mathbb{N}$.
For example, the tree in Figure\ \ref{Example:FirstFiniteTree} is a finite tree with two levels.

Finite trees can also be viewed as $p$-\textit{automatic sequences} (see Definition 1.1 in\ \cite{AlloucheShallit92}). The set of subsequences $(\nu_p\circ f (p^m q+r))$, where $m\geq 0$ and $0\leq r < p^m$, is called the $p$-\textit{kernel} of the sequence $(\nu_p\circ f (n))$. We refer the reader to Allouche and Shallit's book\ \cite{AlloucheShallit} and Bell's paper\ \cite{Bell} for more information.

In infinite trees, an \textit{infinite branch} refers to a path that travels through non-terminating nodes and branches. An infinite tree contains at least one infinite branch. See Figures\ \ref{Example:OneInfinteBranch} and\ \ref{Example:TwoInfinteBranch} in the Appendix for examples of infinite trees.
  
An infinite branch in the tree provides some key insight into the zeros of the polynomial $f$: any infinite branch in the tree associated to the polynomial $f$ corresponds to a root of $f(x) = 0$ in the $p$-adic ring of integers $\mathbb{Z}_p$. This result can be attributed to Byrnes et al.\ \cite{Byrnes}.
We should point out that the case of a tree being finite corresponds to the sequence of $p$-adic valuations being both bounded and periodic. If a tree contains at least one infinite branch, then the sequence of $p$-adic valuations is unbounded. Lastly, we discuss the shape of the trees -- this pertains to the tree structure: which branches terminate, at what level they terminate, at what level the branches split, and the valuations at the terminating nodes.

The following theorem provides the main results of the paper.
\begin{theorem}\label{theorem:summary}
Let $p$ be an odd prime and let $f (n) = a n^{2} + b n + c \in \mathbb{Z} [n]$ such that $p \nmid \gcd(a, b, c)$.
\begin{enumerate}
    \item If $a \equiv 0 \pmod p$ and $b \equiv 0 \pmod p$, then the tree has one node with valuation $\nu_{p} (f (n)) = 0$.
    \item If $a \equiv 0 \pmod p$ and $b \not\equiv 0 \pmod p$, then the tree has one infinite branch.
    \item If $a \not\equiv 0 \pmod p$, then we consider the discriminant $D = b^{2} - 4 a c$.
    \begin{enumerate}
        \item If $D = 0$, then the tree has one infinite branch.
        \item If $D = p^{\nu_{p} (D)} \Delta$ where $\nu_{p} (D)$ is even and $\Delta$ is a quadratic residue modulo $p$, then the tree has two infinite branches.
        \item If $D = p^{\nu_{p} (D)} \Delta$ where $\nu_{p} (D)$ is odd or $\Delta$ is a quadratic non-residue modulo $p$, then the tree is finite with $\ell = \lceil \nu_{p} (D) / 2 \rceil$ levels.
    \end{enumerate}
\end{enumerate}
\end{theorem}
 The result above will be proven by extending the ideas from Boultinghouse et al.\ \cite{Boultinghouse} by using Hensel's lemma and the theory of quadratic residues.   

 The following theorem describes the exact structure of all finite trees from part 3(c) of Theorem\ \ref{theorem:summary} regarding the valuations at each terminating node. 
 \begin{theorem}\label{theorem:FiniteTrees}
Let $p$ be an odd prime and let $f(n)=an^{2}+bn+c \in \mathbb{Z} [n]$ such that $\nu_{p}(a)=0$. Let $D=b^{2}-4ac\neq0$ denote the discriminant. Suppose the discriminant can be expressed as $D=p^{\nu_{p}(D)}\Delta$ so that $p \nmid \Delta$, and the valuation $\nu_{p}(D)$ is odd or $\Delta$ is a quadratic non-residue modulo $p$. Set $\ell=\lceil\nu_{p}(D)/2\rceil$. Then there exists a natural number $S_{\ell}$ with the property
\begin{equation*}
    \nu_{p}\left(\frac{b}{2a}-S_{\ell}\right)\geq\ell,
\end{equation*} 
such that
\begin{equation*}
    \nu_{p}(f(n))=\begin{cases}
    2(m-1),&\ \text{if}\ n\equiv\modd{(p^{m-1}r-S_{\ell})} {p^{m}}\ \text{for}\ m=1,2,\ldots,\ell;\\
    \nu_{p}(D),&\ \text{if}\ n\equiv\modd{-S_{\ell}} {p^{\ell}},\\
    \end{cases}
\end{equation*}
where $r\in\left\{1,2,\ldots,p-1\right\}$.
\end{theorem}
The theorem above will be proven using the translation arguments for the structures of finite trees developed in Boultinghouse et al.\ \cite{Boultinghouse}. 

Note that this theorem can be stated in the context of the tree associated with the function $f(n)$. Consider the assumptions of Theorem\ \ref{theorem:FiniteTrees}. Let $r\in\left\{1,2,\ldots,p-1\right\}$. At a given level $m$, where $1\leq m<\ell$, the tree has $p-1$ terminating nodes at the nodes $n\equiv\modd{(p^{m-1}r-S_{\ell})} {p^{m}}$ with valuation $\nu_{p}(f(n))=2(m-1)$. At the $\ell^{th}$ level, the nodes are terminating. Additionally, each node of the form $n\equiv\modd{(p^{\ell-1}r-S_{\ell})} {p^{\ell}}$ has a valuation of $\nu_{p}(f(n))=2(\ell-1)$ and the node of the form $n\equiv\modd{-S_{\ell}} {p^{m}}$ has a valuation $\nu_{p}(f(n))=\nu_{p}(D)$. See Figure\ \ref{Example:FiniteTree} in the Appendix for such an example.

\section{The \texorpdfstring{$p$}{p}-adic numbers and other preliminaries}

The $p$-adic numbers $\mathbb{Q}_{p}$ are all elements expressible as $x = \sum_{i = m}^{\infty} a_{i} p^{i}$ where $m \in \mathbb{Z} \cup \{ \infty \}$, $a_{i} \in \{ 0, 1, \ldots, p - 1 \}$, and $a_{m} \neq 0$.

With the $p$-adic expansion of $x$ we have the equivalent definition of the $p$-adic valuation as
\begin{equation*}
\nu_{p}(x)=\nu_{p}\left( \sum_{i = m}^{\infty} a_{i} p^{i}\right)=m.
\end{equation*}
By definition we let $\nu_{p}(0)=\infty$. Additionally, for $x,y\in\mathbb{Q}_{p}$ we have the identity
\begin{equation*}
\nu_{p}(xy)=\nu_{p}(x)+\nu_{p}(y).
\end{equation*}

A space of interest will be the ring of $p$-adic integers, denoted as $\mathbb{Z}_{p}$:
\begin{equation*}
\mathbb{Z}_{p}=\left\{x\in\mathbb{Q}_{p}:x=\sum_{i = 0}^{\infty} a_{i} p^{i}\right\}.
\end{equation*}

The next result is a well-known lemma that could be used to determine the existence of zeros in $\mathbb{Z}_{p}$ for polynomials with integer coefficients. We cite the lemma without proof here, as stated in Section 6.4 in the work by Robert\ \cite{Robert}.
\begin{lemma}[Hensel's lemma]\label{lemma:Hensel}
    Let $f (x) \in \mathbb{Z}_{p} [x]$ and suppose there is some $x_{0} \in \mathbb{Z}_{p}$ that satisfies
    \begin{equation*}
        f (x_{0}) \equiv\modd{0} {p^{n}}.
    \end{equation*}
If $\phi = \nu_{p} (f' (x_{0})) < \frac{n}{2}$, then there exists a unique zero $\xi$ of $f (x)$ in $\mathbb{Z}_{p}$ such that
    \begin{equation*}
        \xi \equiv\modd{x_{0}}  {p^{n - \phi}} \qquad \text{and} \qquad \nu_{p} (f' (\xi)) = \nu_{p} (f' (x_{0})) = \phi.
    \end{equation*}
\end{lemma}
The conditions on $\xi$ imply that only one zero can exist with those particular conditions. However, observe that there can exist more than one zero in general.

\begin{lemma}[Theorem 2.6 Byrnes et al.\ \cite{Byrnes}]\label{Byrnes:Theorem2.6}
Any infinite branch in the tree associated to the polynomial $f(n)$ corresponds to a root of $f(x)=0$ in the $p$-adic field $\mathbb{Q}_{p}$.
\end{lemma}

\begin{lemma}[Theorem 2.1 Medina et al.\ \cite{Medina}]\label{Medine:Theorem2.1}
Let $p$ be a prime number and $f\in\mathbb{Z}[x]$. Then $(\nu_{p}(f(n)))$ is either periodic or unbounded. Moreover, $(\nu_{p}(f(n)))$ is periodic if and only if $f(n)$ has no zeros in $\mathbb{Z}_{p}$. In the periodic case, the minimal period length is a power of $p$.
\end{lemma}

From the proof of Theorem 4.1 by Medina et al.\ \cite{Medina}, we have the following result regarding the existence of square roots in $\mathbb{Z}_{p}$.
\begin{proposition}\label{theorem:sqrtquadres}
    Let $p$ be an odd prime, and let $a \in \mathbb{Z}$. For nonzero $a$, let $d$ be the unique integer which satisfies $a = p^{\nu_{p} (a)} \cdot d$. Then there exists a square root of $a$ in $\ZZ_p$ if and only if $\nu_{p} (a)$ is even and $d$ is a quadratic residue modulo $p$.
\end{proposition}

Determining whether a tree associated with the quadratic function is finite or infinite is tied to the zeros of that quadratic function being in $\mathbb{Z}_{p}$. The zeros of the quadratic depend on the square root of the discriminant being in $\mathbb{Z}_{p}$. Hence, there is a need to work with quadratic residues. Next, we present the definition and some basic facts.
\begin{definition}
A number $r \in \mathbb{Z}$ is a \textit{quadratic residue modulo} $p$ if there exists $x \in \mathbb{Z}$ such that $x^{2} \equiv r \pmod p$.
\end{definition}

Next, we prove some simple ideas involving quadratic residues that will be useful later in the paper.
\begin{lemma}\label{lemma:evenvalsqrt}
Let $a \in \mathbb{Q}_{p}$. If there exists $x \in \mathbb{Q}_{p}$ such that $x^{2} = a$, then $\nu_{p} (a) = 2 k$ for some $k \in \mathbb{Z}$.
\end{lemma}
\begin{proof}
Let $a \in \mathbb{Q}_{p}$ and suppose there exists $x \in \mathbb{Q}_{p}$ such that $x^{2} = a$. Then $\nu_{p} (x^{2}) = 2 \nu_{p} (x) = \nu_{p} (a)$. Since $p$-adic valuations must be integers, then $\nu_{p} (a) = 2 k$ for some $k \in \mathbb{Z}$.
\end{proof}
\begin{lemma}\label{lemma:sqrtZp}
Let $a \in \mathbb{Z}_{p}$. If there exists $x \in \mathbb{Q}_{p}$ such that $x^{2} = a$, then $x \in \mathbb{Z}_{p}$.
\end{lemma}

\begin{proof}
Let $a \in \mathbb{Z}_{p}$. Suppose there exists $x \in \mathbb{Q}_{p}$ such that $x^{2} = a$. From Lemma\ \ref{lemma:evenvalsqrt}, we have $2 \nu_{p} (x) = \nu_{p} (a)$. Since $\nu_{p} (a) \geq 0$, it follows that $\nu_{p} (x) \geq 0$ and thus $x \in \mathbb{Z}_{p}$.
\end{proof}

In conjunction with quadratic residues, we will later make use of the Legendre symbol. We define the \textit{Legendre symbol}, denoted by $\left(\frac{a}{p}\right)$, as a function of an integer $a$ and an odd prime $p$ such that
\begin{equation*}
\left(\frac{a}{p}\right)=\begin{cases}
1,&\ \text{if}\  a\, \text{is a quadratic residue modulo}\ p\ \text{and}\ a\not\equiv\modd{0} {p};\\
-1,&\ \text{if}\ a\, \text{is not a quadratic residue modulo}\ p;\\
0,&\ \text{if}\ a\equiv\modd{0} {p}.
\end{cases}
\end{equation*}

Note the Legendre symbol can be calculated by
\begin{equation*}
\left(\frac{a}{p}\right)\equiv\modd{a^{\frac{p-1}{2}}} {p}\ \text{and}\ \left(\frac{a}{p}\right)\in\left\{-1,0,1\right\}.
\end{equation*}

In the following lemma, we state several other properties of the Legendre symbol.
\begin{lemma} Let $p$ be an odd prime and $a$, $b$ be any integers. Then:
\begin{enumerate}
    \item If $a\equiv\modd{b} {p}$, then $\left(\frac{a}{p}\right)=\left(\frac{b}{p}\right)$.
    \item The Legendre symbol has the multiplicative property. That is, $\left(\frac{a}{p}\right)\left(\frac{b}{p}\right)= \left(\frac{ab}{p}\right)$.
\end{enumerate}
\end{lemma}

To determine the valuations, we frequently use properties of primes and polynomials. To simplify calculations, we mention a property of considering polynomials modulo a power of a prime $p$, extending the property introduced by Byrnes et al.\ \cite{Byrnes} when the power of $p$ is $1$.
\begin{proposition}\label{proposition:fmodp}
Let $f (n)$ be a polynomial with integer coefficients, and let $p$ be a prime. Then for $q, r \in \mathbb{Z}$ with $i \geq 1$, we have $f (p^{i} q + r) \equiv\modd{f(r)} {p^{i}}$.
\end{proposition}
\begin{proof}
Let $f (n) = a_{m} n^{m} + \cdots + a_{1} n + a_{0} \in \mathbb{Z} [n]$. Then applying the binomial theorem gives
\begin{equation*}
    \begin{split}
        f (p^{i} q + r) &= a_{m} (p^{i} q + r)^{m} + a_{m - 1} (p^{i} q + r)^{m - 1} + \cdots + a_{0} \\
        &= a_{m} \sum_{k = 0}^{m} \binom{m}{k} (p^{i} q)^{k} r^{m - k} + a_{m - 1} \sum_{k = 0}^{m - 1} \binom{m - 1}{k} (p^{i} q)^{k} r^{m - 1 - k} + \cdots + a_{0} \\
        &\equiv\modd{a_{m} r^{m} + a_{m - 1} r^{m - 1} \cdots + a_{0}} {p^{i}} \\
        &\equiv\modd{f(r)} {p^{i}}.
    \end{split}
\end{equation*}
In other words, when considering a polynomial $f (n) \in \mathbb{Z} [n]$ modulo $p^{i}$, the input can be reduced modulo $p^{i}$ before proceeding with the evaluation of the polynomial.
\end{proof}

\section{The case for \texorpdfstring{$a \equiv \modd{0} {p}$}{a equiv modd{0} {p}} }

In this section, we will prove parts (1) and (2) from Theorem\ \ref{theorem:summary}. We need the following lemma to prove part (1) of Theorem\ \ref{theorem:summary}.
\begin{lemma}\label{proposition:dotTrees}
    Let $p$ be a prime and let $f(n)=an^2+bn+c \in \mathbb{Z} [n]$. Then $\nu_p(f(n))=0$ if and only if $f(r) \not \equiv 0 \pmod p$ holds for all $r \in \{0, 1, \ldots, p-1\}$.
\end{lemma}
\begin{proof}
    Let $p$ be a prime and let $f(n)=an^2+bn+c \in \mathbb{Z}[n]$. Then
    $\nu_p(f(n))=0$ if and only if 
  $\nu_p(f(pq+r))=0$ for all $r \in \{0, 1, \ldots, p-1\}$. From Proposition\ \ref{proposition:fmodp}, we see that $f (p q + r) \equiv f (r) \pmod p$. Therefore, $\nu_p(f(n))=0$ if and only if $f(r) \not \equiv 0 \pmod p$ holds for all $r \in \{0, 1, \ldots, p-1\}$.
\end{proof}

\begin{proposition}\label{Proprosition:InfiniteTree}
Let $p$ be an odd prime and let $f (n) = a n^{2} + b n + c \in \mathbb{Z} [n]$. If $a \equiv 0 \pmod p$, $b \equiv 0 \pmod p$, and $c \not\equiv 0 \pmod p$, then the corresponding tree is a single node with valuation zero.
\end{proposition}
\begin{proof}
Let $p$ be an odd prime. Let $f(n)=an^2+bn+c \in \mathbb{Z}[n]$ such that $a \equiv 0 \pmod{p}$, $b \equiv 0 \pmod{p}$, and $c \not \equiv 0 \pmod{p}$.  This implies that $a=p\alpha$ and $b=p\beta$ for some integers $\alpha$ and $\beta$.  Therefore for all $r \in \{0, 1, \ldots, p-1\}$, we have that $f(r)=p\alpha r^2+p\beta r+c$. Since $c \not \equiv 0 \pmod{p}$, then $f(r)=p\alpha r^2+p\beta r+c \not \equiv 0 \pmod{p}$.  Therefore by Lemma\ \ref{proposition:dotTrees}, the tree of $f(n)$ is a dot tree with a valuation of 0.
\end{proof}

The trees described in the above proposition are the simplest types of trees one can have, and we will refer to them as \textit{dot trees}. The dot trees in Proposition\ \ref{Proprosition:InfiniteTree} have only the valuation of zero. However, dot trees may have valuations other than zero. This case occurs when the quadratic coefficients are all divisible by $p$ to some power. We can state this formally as the following corollary.

\begin{corollary}
Let $p$ be a prime and let $f (n) = a n^{2} + b n + c$. Then the corresponding tree has a single valuation of $k \geq 0$ if and only if there exists a $k\in\mathbb{N}$ such that $p^{k}$ divides $a$, $b$, and $c$, but $p^{k+1}$ divides $a$ and $b$ but not $c$.
\end{corollary}

Next, we prove part (2) of Theorem\ \ref{theorem:summary}. Again, we first need a lemma. As noted above, the tree associated with a quadratic is infinite if a zero of the quadratic is in $\mathbb{Z}_{p}$. The nature of these zeros is tied to the discriminant of the quadratic. Furthermore, we see that the $p$-adic valuation of the discriminant plays a role in determining when the tree associated to a quadratic is finite or infinite. The next lemma is used to simplify a calculation.

\begin{lemma}\label{lemma:D0}
Let $p$ be a prime and let $D = b^{2} - 4 a c$ where $a\equiv\modd{a_{0}} {p}$, $b\equiv\modd{b_{0}} {p}$, and $c\equiv\modd{c_{0}} {p}$ with $a_{0}, b_{0}, c_{0} \in \{ 0, 1, \ldots, p - 1 \}$. Then $\nu_{p} (D) = 0$ if and only if $\nu_{p} (b_{0}^{2} - 4 a_{0} c_{0}) = 0$.
\end{lemma}
\begin{proof}
Suppose $\nu_{p}(D)=0$. Notice that from a straightforward calculation, $D$ is of the form $D=pt+b_{0}^{2}-4a_{0}c_{0}$, where $t\in\mathbb{Z}$. Then $D$ is not divisible by $p$ if and only if $b_{0}^{2}-4a_{0}c_{0}$ is not divisible by $p$.
\end{proof}

\begin{proposition}\label{Proposition:InfiniteTreeOne}
Let $p$ be an odd prime and let $f (n) = a n^{2} + b n + c \in \mathbb{Z} [n]$. If $a\equiv\modd{0} {p}$ and $b \not\equiv\modd{0} {p}$, then the corresponding tree is infinite with exactly one infinite branch.
\end{proposition}
\begin{proof}
We first show that the corresponding tree has at least one infinite branch, and then we show that it has at most one infinite branch to reach the desired result of exactly one infinite branch.

To determine that there is at least one infinite branch, we apply Hensel's lemma (Lemma\ \ref{lemma:Hensel}). Let $f(x) = a x^{2} + b x + c$. Taking the derivative gives $f'(x) = 2 a x + b$. We want to find some $x_{0} \in \mathbb{Z}_{p}$ so that $f (x_{0}) \equiv 0 \pmod{p^{n}}$ and $\nu_{p} (f' (x_{0})) < \frac{n}{2}$.

Let $x_{0} = -c b^{p - 2} \in \mathbb{Z}$. Recalling $a \equiv 0 \pmod p$ and applying Fermat's Little Theorem, we have
\begin{equation*}
    f (x_{0}) = a (-c b^{p - 2})^{2} + b (-c b^{p - 2}) + c \equiv -c b^{p - 1} + c \equiv 0 \pmod p,
\end{equation*}
so that $n \geq 1$. Then
\begin{equation*}
    f' (x_{0}) = 2 a (-c b^{p - 2}) + b \equiv b \pmod p,
\end{equation*}
which implies $\phi = \nu_{p} (f' (x_{0})) = 0$. Since $\phi < \frac{n}{2}$, Hensel's lemma guarantees the existence of a zero. Therefore the tree has at least one infinite branch.

We apply Lemma\ \ref{lemma:D0} to determine that there is exactly one infinite branch and exactly one zero of $f (n)$ is in $\mathbb{Q}_{p}$ and is not in $\mathbb{Z}_{p}$. That is, we determine that exactly one zero has a negative $p$-adic valuation. Consider the zeros $n_{1, 2} = \frac{-b \pm \sqrt{D}}{2 a}$, where $D = b^{2} - 4 a c = p^{\nu_{p} (D)} \Delta$. By hypothesis we have $a = p \alpha$, $b = p \beta + b_{0}$, and $c = p \gamma + c_{0}$ with $b_{0} \in \{ 1, 2, \ldots, p - 1 \}$, and $c_{0} \in \{ 0, 1, \ldots, p - 1 \}$. Then it follows from Lemma\ \ref{lemma:D0} that $\nu_{p}(D)=0$, since $\nu_{p}(b_{0}^{2})=0$. Furthermore, 
\begin{equation*}
    D = p (p (\beta^{2} - 4 \alpha \gamma) - 2 (2 \alpha c_{0} - \beta b_{0})) + b_{0}^{2} \equiv b_{0}^{2} \pmod p.
\end{equation*}
It follows that $D = \Delta \equiv b_{0}^{2} \pmod p$ is a quadratic residue. Since the valuation of the discriminant is even and $\Delta$ is a quadratic residue, we have $\sqrt{D} \in \mathbb{Z}_{p}$ by Proposition\ \ref{theorem:sqrtquadres}. Furthermore, neither $b$ nor $\sqrt{D}$ is congruent to zero modulo $p$. Since $\sqrt{D}\not\equiv0\pmod p$, it follows that $2\sqrt{D}\not\equiv0\pmod p$ and therefore $\sqrt{D} \not\equiv -\sqrt{D} \pmod p$. 

Furthermore, at least one $-b - \sqrt{D}$ or $-b + \sqrt{D}$ cannot be congruent to zero modulo $p$. If both were congruent to zero mod $p$, we could combine the two statements of $-b - \sqrt{D}\equiv0\pmod p$ and $-b + \sqrt{D}\equiv0\pmod p$ to get $\sqrt{D}\equiv-\sqrt{D}\pmod p$, which is a contradiction. Without loss of generality, suppose $-b - \sqrt{D} \not\equiv 0 \pmod p$. This implies $\nu_{p} \left(-b - \sqrt{D}\right) = 0$. Recalling $\nu_{p} (a) \geq 1$, we have
\begin{equation*}
    \nu_{p} (n_{1}) = \nu_{p} \left( \frac{-b - \sqrt{D}}{2 a} \right) = \nu_{p} \left(-b - \sqrt{D}\right) - \nu_{p} (2 a) = 0 - \nu_{p} (a) < 0,
\end{equation*}
and thus $n_{1} \in \mathbb{Q}_{p} \setminus \mathbb{Z}_{p}$. It follows that at least one zero is not in $\mathbb{Z}_{p}$, and thus the tree has at most one infinite branch.
\end{proof}

With a little more work, we can have the following corollary.
\begin{corollary} If $a\equiv \modd{0} {p}$ and $b\not\equiv\modd{0} {p}$, then on the $m^{th}$ level, there is only one non-terminal node with a valuation greater than or equal to $m$. The other $p-1$ nodes are terminal with valuation equal to $m-1$.
\end{corollary}

\begin{proof}
From Proposition\ \ref{Proposition:InfiniteTreeOne} there must exist at least one non-terminating node at $n=pq+r_{0}$, where $r_{0}\in\left\{0,1,\ldots,p-1\right\}$, such that $p\mid f(pq+r_{0})$. Furthermore, a calculation yields
\begin{equation*}
a(pq+r_{0})^{2}+b(pq+r_{0})+c=p(apq^{2}+bq+2aqr_{0}+r_{0}^{2}a/p)+br_{0}+c,
\end{equation*}
thus $p\mid (br_{0}+c)$. That is, 
\begin{equation*}
br_{0}+c=pk_{1}\ \text{for}\ k_{1}\in\mathbb{Z}.
\end{equation*}
Thus, the valuation at this node is greater than or equal to $1$. 

Additionally, suppose there exists a second non-terminating node at $n=pq+r_{0}'$, where $r_{0}'\in\left\{0,1,\ldots,p-1\right\}\setminus\left\{r_{0}\right\}$, such that $p\mid f(pq+r_{0}')$. A similar calculation shows that $p\mid (br_{0}'+c)$ as well. That is, 
\begin{equation*}
br_{0}'+c=pk_{2}\ \text{for}\ k_{2}\in\mathbb{Z}.
\end{equation*}
Taking the difference of $br_{0}+c=pk_{1}$ and $br_{0}'+c=pk_{2}$, we get
\begin{equation*}
b(r_{0}-r_{0}')=p(k_{1}-k_{2}).
\end{equation*}
Since $r_{0},r_{0}'\in\left\{0,1,\ldots,p-1\right\}$, it follows that $p\nmid (k_{1}-k_{2})$. But then $p\mid b$, which is a contradiction. Thus there exists exactly one non-terminating node. Furthermore, since
\begin{equation*}
a(pq+r_{0}')^{2}+b(pq+r_{0}')+c=p(apq^{2}+bq+2aqr_{0}'+r_{0}'^{2}a/p)+br_{0}'+c,
\end{equation*}
for all $r_{0}'\in\left\{0,1,\ldots,p-1\right\}\setminus\left\{r_{0}\right\}$, the valuations at all terminating nodes must be zero.

Suppose by induction for all levels $i$ up to $m-1$ with $m\geq2$ the terminating nodes have a valuation equal to $i-2$ and there is only one non-terminating node, with a valuation greater than or equal to $i-1$.

Now suppose the only non-terminating node is at $n=p^{m-1}q+r_{0}$, and so $p^{m-1}\mid f(p^{m-1}q+r_{0})$. This node splits into the nodes of the form $n=p^{m}q+p^{m-1}r+r_{0}$, where $r\in\left\{0,1,\ldots,p-1\right\}$. Then 
\begin{align*}
f(p^{m}q+p^{m-1}r+r_{0})=a(p^{m}q+r_{0})^{2}&+b(p^{m}q+r_{0})+c\\
&+2ap^{m-1}rr_{0}+2ap^{2m-1}qr+ap^{2(m-1)}r^{2}+bp^{m-1}r.
\end{align*}
 Without loss of generality, let the non-terminating node be at $n=p^{m}q+r_{0}$, that is $r=0$. Since $n=p^{m}q+r_{0}$ is the non-terminating node, and we know the valuations are bigger than or equal to $m-1$, the above implies 
\begin{align*}
f(p^{m}q+r_{0})&=a(p^{m}q+r_{0})^{2}+b(p^{m}q+r_{0})+c\\
&=p^{m- 1}(ap^{m+1}q^{2}+2apqr_{0}+bpq)+ar_{0}^{2}+br_{0}+c\\
&=p^{m-1}k_{3}\ \text{for}\ k_{3}\in\mathbb{Z}\ \text{depending on}\ q.
\end{align*}
Additionally, since $n=p^{m}q+r_{0}$ is the non-terminating node and $p^{m-1}\mid f(p^{m-1}q+r_{0})$, this implies $p^{m-1}\mid (ar_{0}^{2}+br_{0}+c)$. But if $p^{m}\nmid(ar_{0}^{2}+br_{0}+c)$, then the node $n=p^{m}q+r_{0}$ would have constant valuation equal to $m-1$, which contradicts the fact that this node is non-terminating. Thus we must have $p^{m}\mid f(p^{m-1}q+r_{0})$, which implies that $p\mid k_{3}$.

This fact, and the equation above, then imply
\begin{align*}
f(p^{m}q+p^{m-1}r+r_{0})&=p^{m-1}k_{3}+2ap^{m-1}rr_{0}+2ap^{2m-1}qr+ap^{2(m-1)}r^{2}+bp^{m-1}r\\
&=p^{m-1}(k_{3}+2arr_{0}+2ap^{m}qr+ap^{m-1}r^{2}+br),
\end{align*}
for all other $r\in\left\{1,\ldots,p-1\right\}$. Since $b$ and $r$ are not divisible by $p$, then $br$ is not divisible by $p$. Thus all other nodes are terminal with valuation $m-1$ on the $m^{th}$ level. This then completes the proof.
\end{proof}

\section{The case for \texorpdfstring{$a \not\equiv \modd{0} {p}$}{a not equiv modd{0} {p}}}
In this section, we first prove the case (3) from Theorem\ \ref{theorem:summary} and then we prove Theorem\ \ref{theorem:FiniteTrees}. The following proposition comes in three parts. The first two parts prove cases (3a) and (3b) of Theorem\ \ref{theorem:summary}. The third part proves case (3c). Then, with the help of Corollary\ \ref{corollary:translation}, we prove Theorem\ \ref{theorem:FiniteTrees}. 

\begin{proposition}\label{Theorem:FiniteTrees}
Let $p$ be an odd prime and let $f(n)=an^{2}+bn+c\in\mathbb{Z}[n]$ such that $\nu_{p}(a)=0$. Let $D=b^{2}-4ac$ denote the discriminant and express $D=p^{\nu_{p}(D)}\Delta$ so that $p\nmid\Delta$.
\begin{enumerate}
    \item If $D=0$, then the tree has exactly one infinite branch corresponding to the zero $n_{1}=-\frac{b}{2a}$.
    \item For $D\neq0$, if $\nu_{p}(D)$ is even and $\Delta$ is a quadratic residue modulo $p$, then the tree has exactly two infinite branches corresponding to the zeros $n_{1,2}=\frac{-b\pm\sqrt{D}}{2a}$.
    \item For $D\neq0$, if $\nu_{p}(D)$ is odd or $\Delta$ is a quadratic non-residue modulo $p$, then the tree is finite with $\ell=\lceil\nu_{p}(D)/ 2\rceil$ levels. Additionally, there exists a number $S_{\ell}\in\mathbb{Z}$ such that the tree associated to $f(n-S_{\ell})$ has the following properties:
    \begin{enumerate}
        \item The number of levels is equal to $\ell=\lceil\nu_{p}(D)/ 2\rceil$.
        \item Other than the last level, the node of the form $n=p^{j}q$ where $1\leq j\leq \ell-1$ is the only non-terminating node.
        \item The valuation of the terminating nodes is equal to $2(m-1)$ where $m$ is the level. The only exception is the node $n=p^{\ell}q$ on the last level.
        \item The node $n=p^{\ell}q$ on the last level has valuation $\nu_{p}(D)$. This is the maximum valuation of the tree.
    \end{enumerate}
\end{enumerate}
\end{proposition}
\begin{proof}
Part (1) is trivial with $D=0$.

For part (2), consider $D=p^{\nu_{p}(D)}\Delta$ with $\nu_{p}(D)$ even and $\Delta$ a quadratic residue modulo $p$. By Proposition\ \ref{theorem:sqrtquadres}, $\sqrt{D}$ is a $p$-adic integer.  Since $\nu_{p}(2a)=0$, then $n_{1,2}=\frac{-b\pm\sqrt{D}}{2a}$ are $p$-adic integers as well. 

For part (3), define $S$ to equal in its $p$-adic expansion to
\begin{equation*}
S=\frac{b}{2a}=\sum_{i=0}^{\infty}s_{i}p^{i},
\end{equation*}
where $s_{i}\in\left\{0,1,2,\ldots,p-1\right\}$.

Define $S_{\ell}$ as 
\begin{equation*}
S_{\ell}=\sum_{i=0}^{\ell-1}s_{i}p^{i},
\end{equation*}
which is the truncation of $S$. Note that by definitions of $S$ and $S_{\ell}$ from above we have
\begin{align*}
\nu_{p}\left(\frac{b}{2a}-S_{\ell}\right)&=\nu_{p}(b-2aS_{\ell})\\
&\geq\ell.
\end{align*}
Then
\begin{equation*}
b-2aS_{\ell}=p^{\ell}k_{1},
\end{equation*}
where $k_{1}\in\mathbb{Z}$. A simple calculation gets us
\begin{align*}
f(n-S_{\ell})&=a(n-S_{\ell})^{2}+b(n-S_{\ell})+c\\
&=an^{2}+(b-2aS_{\ell})n+aS_{\ell}^{2}-bS_{\ell}+c\\
&=an^{2}+p^{\ell}k_{1}n+aS_{\ell}^{2}-bS_{\ell}+c\\
&=an^{2}+p^{\ell}k_{1}n+f(-S_{\ell}).
\end{align*}
The discriminant of $f(n)$ is equal to the discriminant of $f(n-\tau)$, where $\tau\in\mathbb{Z}$. Thus we have
\begin{equation*}
p^{2\ell}k_{1}^{2}-4af(-S_{\ell})=p^{\nu_{p}(D)}\Delta.
\end{equation*}
Since $\ell=\lceil\nu_{p}(D)/2\rceil$, then $2\ell\geq\nu_{p}(D)$ (that is, either $2\ell=\nu_{p}(D)$ or $2\ell=\nu_{p}(D)+1$). It follows that $p^{\nu_{p}(D)}$ divides $f(-S_{\ell})$ since $4a$ and $\Delta$ are not divisible by $p$. Then we have
\begin{equation*}
f(-S_{\ell})=p^{\nu_{p}(D)}k_{2},
\end{equation*} 
where $k_{2}\in\mathbb{Z}$ is not divisible by $p$. Now we have
\begin{equation*}
f(n-S_{\ell})=an^{2}+p^{\ell}k_{1}n+p^{\nu_{p}(D)}k_{2}.
\end{equation*}

At the $m^{th}$ level the nodes are of the form $n=p^{m}j+p^{m-1}r$, where $1\leq m\leq\ell$ and $r\in\left\{0,1,2,\ldots,p-1\right\}$. Then
\begin{align*}
f(n-S_{\ell})&=f(p^{m}j+p^{m-1}r-S_{\ell})\\
&=a(p^{m}j+p^{m-1}r)^{2}+p^{\ell}k_{1}(p^{m}j+p^{m-1}r)+p^{\nu_{p}(D)}k_{2}\\
&=ap^{2m}j^{2}+2ap^{2m-1}rj+ap^{2m-2}r^{2}+p^{\ell+m}k_{1}j+p^{\ell+m-1}k_{1}r+p^{\nu_{p}(D)}k_{2},
\end{align*}
when $r\neq 0$ and 
\begin{equation*}
f(n-S_{\ell})=ap^{2m}j^{2}+p^{\ell+m}k_{1}j+p^{\nu_{p}(D)}k_{2},
\end{equation*}
when $r=0$.

At a given level $m<\ell$, if $r\neq0$, then notice that $2m-2$ is the smallest exponent on $p$. We get
\begin{equation*}
f(n-S_{\ell})=p^{2m-2}(ap^{2}j^{2}+2aprj+ar^{2}+p^{\ell-m+2}k_{1}j+p^{\ell-m+1}k_{1}r+p^{\nu_{p}(D)-2m+2}k_{2}).
\end{equation*}
But notice that $ar^{2}$ is not divisible by $p$. Thus we have
\begin{equation*}
\nu_{p}(f(n-S_{\ell}))=2m-2.
\end{equation*}
If $r=0$, then $2m$ is the smallest exponent on $p$. We get
\begin{equation*}
f(n-S_{\ell})=p^{2m}(aj^{2}+p^{\ell-m}k_{1}j+p^{\nu_{p}(D)-2m}k_{2}).
\end{equation*}
Since $j$ is a natural number, then
\begin{equation*}
\nu_{p}(f(n-S_{\ell}))\geq2m.
\end{equation*}
In other words, at every level $m<\ell$, there is only one non-terminal node which is of the form $n=p^{m}j$, and all other nodes are terminating with $p$-adic valuation equal to $2m-2$.

At the $\ell$th level, our equations will now take the form
\begin{equation*}
f(n-S_{\ell})=ap^{2\ell}j^{2}+2ap^{2\ell-1}rj+ap^{2\ell-2}r^{2}+p^{2\ell}k_{1}j+p^{2\ell-1}k_{1}r+p^{\nu_{p}(D)}k_{2},
\end{equation*}
when $r\neq 0$, and 
\begin{equation*}
f(n-S_{\ell})=ap^{2\ell}j^{2}+p^{2\ell}k_{1}j+p^{\nu_{p}(D)}k_{2},
\end{equation*}
when $r=0$.

When $r\neq0$, then $2\ell-2$ is the smallest exponent on $p$. We get
\begin{equation*}
f(n-S_{\ell})=p^{2\ell-2}(ap^{2}j^{2}+2aprj+ar^{2}+p^{2}k_{1}j+pk_{1}r+p^{\nu_{p}(D)-2\ell+2}k_{2}).
\end{equation*}
Since $ar^{2}$ is not divisible by $p$, we have
\begin{equation*}
\nu_{p}(f(n-S_{\ell}))=2\ell-2.
\end{equation*}

When $r=0$, consider two cases: either $\nu_{p}(D)$ is odd or $\nu_{p}(D)$ is even. If $\nu_{p}(D)$ is odd, then $2\ell>\nu_{p}(D)$. Thus, we have
\begin{equation*}
f(n-S_{\ell})=p^{\nu_{p}(D)}(apj^{2}+pk_{1}j+k_{2}).
\end{equation*}
Recall that $k_{2}$ is not divisible by $p$, and so we have
\begin{equation*}
\nu_{p}(f(n-S_{\ell}))=\nu_{p}(D).
\end{equation*}

If  $\nu_{p}(D)$ is even, then $2\ell=\nu_{p}(D)$. Then we have
\begin{equation*}
f(n-S_{\ell})=p^{\nu_{p}(D)}(aj^{2}+k_{1}j+k_{2}).
\end{equation*}
Recall that $\nu_{p}(a)=\nu_{p}(k_{2})=0$ and $\nu_{p}(k_{1})\geq0$. Then we have
\begin{align*}
\nu_{p}(aj^{2}+k_{1}j+k_{2})&=\nu_{p}(4a^{2}j^{2}+4ak_{1}j+4ak_{2})\\
&=\nu_{p}((2aj+k_{1})^{2}-k_{1}^{2}+4ak_{2})\\
&=\nu_{p}((2aj+k_{1})^{2}-\Delta).
\end{align*}
Assume there exists an $r\in\left\{0,1,\ldots,p-1\right\}$ such that when $j=pk+r$, then
\begin{equation*}
\nu_{p}((2a(pk+r)+k_{1})^{2}-\Delta)\geq1,
\end{equation*}
which implies
\begin{equation*}
(2a(pk+r)+k_{1})^{2}-\Delta=p\kappa,
\end{equation*}
where $\kappa\in\mathbb{Z}$. Then it follows that
\begin{equation*}
(2a(pk+r)+k_{1})^{2}-\Delta\equiv( 2ar+k_{1})^{2}-\Delta \pmod p\equiv 0 \pmod p.
\end{equation*}
Thus
\begin{equation*}
(2ar+k_{1})^{2}\equiv\Delta \pmod p.
\end{equation*}
This implies that $\Delta$ is a quadratic residue, which contradicts the assumption that $\Delta$ is not a quadratic residue in the case of $\nu_{p}(D)$ even. Therefore,
\begin{equation*}
\nu_{p}(aj^{2}+k_{1}j+k_{2})=0
\end{equation*}
for all $j\in\mathbb{N}$ and this completes the proof.
\end{proof}

\begin{remark}
With the above assumption, if we have an infinite tree, then we must have that $\nu_{p} (D)$ is even and $\Delta$ is a quadratic residue modulo $p$. To see this, consider the case when either $\nu_{p} (D)$ is odd or $\Delta$ is not a quadratic residue modulo $p$. Then $\sqrt{D}$ does not exist in $\mathbb{Z}_p$, hence neither of $n_{1, 2}$ are $p$-adic integers. Therefore, the corresponding tree is finite.
\end{remark}

\begin{lemma}[Translation]\label{lemma:translation}
    Let $p$ be an odd prime and let $f(n) \in \mathbb{Z}[n]$. Suppose $s\in\mathbb{Z}$, and set $g(n)=f(n-s)$. Given $m\in\mathbb{N}$ and $r\in\mathbb{Z}$ such that $0\leq r<p^{m}$, it follows that
    \begin{equation*}
        \nu_p(f(p^{m}q+r))=\nu_p(g(p^{m}q+\modd{(r+s)} {p^{m}})).
    \end{equation*}
\end{lemma}
\begin{proof}
   If the node $n=p^kq+d$ terminates, the valuation $\nu_p(f(n))$ at the node of the form $n=p^kq+r$ is moved to the node of the form $n=p^kq+\modd{(r+s)} {p^k}$ under translating an equation from $f(n)$ to $f(n-s)=g(n)$.  Similarly, if the node $n=p^kq+d$ continues under $f(n)$, the node of the form $n=p^kq+\modd{(r+s)} {p^k}$ continues under translating an equation from $f(n)$ to $f(n-s)=g(n)$.
\end{proof}

\begin{corollary}\label{corollary:translation}
Let $p$ be an odd prime and let $f(n)=an^{2}+bn+c \in \mathbb{Z} [n]$ such that $\nu_{p}(a)=0$. Let $D=b^{2}-4ac\neq0$ denote the discriminant. Suppose the discriminant can be expressed as $D=p^{\nu_{p}(D)}\Delta$ so that $p \nmid \Delta$, and the valuation $\nu_{p}(D)$ is odd or $\Delta$ is a quadratic non-residue modulo $p$. Set $\ell=\lceil\nu_{p}(D)/2\rceil$. Define a natural number $S_{\ell}$ as in the proof of Proposition\ \ref{Theorem:FiniteTrees}. Then
\begin{equation*}
    \nu_{p}(f(n))=\begin{cases}
    2(m-1),&\ \text{if}\ n\equiv\modd{(p^{m-1}r-S_{\ell})} {p^{m}}\ \text{for}\ m=1,2,\ldots,\ell;\\
    \nu_{p}(D),&\ \text{if}\ n\equiv\modd{-S_{\ell}} {p^{\ell}}.\\
    \end{cases}
\end{equation*}
\end{corollary}
\begin{proof}
This follows immediately from applying Lemma\ \ref{lemma:translation} to Proposition\ \ref{Theorem:FiniteTrees}.
\end{proof}

Additionally, with a little more work we can get the following corollary from Proposition\ \ref{Theorem:FiniteTrees} part 1.
\begin{corollary}
If $a\not\equiv \modd{0} {p}$ and $D=b^{2}-4ac=0$, then on the $m^{th}$ level, the terminal nodes have valuations equal to $2(m-1)$ and the non-terminal nodes have valuation greater than or equal to $2m$.
\end{corollary} 
\begin{proof}
First, note that
\begin{align*}
\nu_{p}(an^{2}+bn+c)&=\nu_{p}(4a^{2}n^{2}+4abn+4ac)\\
&=\nu_{p}(4a^{2}n^{2}+4abn+b^{2}-b^{2}+4ac)\\
&=\nu_{p}((2an+b)^{2}-b^{2}+4ac)\\
&=\nu_{p}((2an+b)^{2})\\
&=2\nu_{p}((2an+b)).
\end{align*}
By Proposition\ \ref{Theorem:FiniteTrees} part 1, there must exist at least one non-terminal node at $n=pq+r_{0}$, where $r_{0}\in\left\{0,1,\ldots,p-1\right\}$, such that $p \mid (2a(pq+r_{0})+b)$. Thus the valuation at this node is greater than or equal to two by the last line of the above equality. Now suppose there exists a second non-terminating node $r_{0}'\in\left\{0,1,\ldots,p-1\right\}\backslash\left\{0\right\}$. Thus $p\mid (2a(pq+r_{0}')+b)$. From these two statements, we have
\begin{equation*}
2a(pq+r_{0})+b=pk_{1}\ \text{and}\ 2a(pq+r_{0}')+b=pk_{2},
\end{equation*}
for some $k_{1},k_{2}\in\mathbb{Z}$ which depend on $q$. Taking the difference, we have
\begin{equation*}
2a(r_{0}-r_{0}')=p(k_{1}-k_{2}).
\end{equation*}
Since $p\nmid(r_{0}-r_{0}')$, we have $p\mid a$, which is a contradiction to the hypothesis. Thus all other nodes must be terminating with a valuation equal to zero.

Suppose by induction for all levels $i$ up to $m-1$ with $m\geq2$ the terminating nodes have a valuation equal to $2(i-2)$. Furthermore, suppose there is only one non-terminating node. This node has valuation greater than or equal to $2(i-1)$.

Suppose the only non-terminating node is at $n=p^{m-1}q+r_{0}$, and so $p^{m-1}\mid (2a(p^{m-1}q+r_{0})+b)$. This node then splits into nodes of the form $n=p^{m}q+p^{m-1}r+r_{0}$, where $r\in\left\{0,1,2,\ldots,p-1\right\}$. Without loss of generality, suppose the non-terminating node is at $n=p^{m}q+r_{0}$, that is, $r=0$. Then a calculation yields
\begin{equation*}
2a(p^{m}q+p^{m-1}r+r_{0})+b=2ap^{m}q+2ap^{m-1}r+2ar_{0}+b.
\end{equation*}
Since $n=p^{m}q+r_{0}$ is the non-terminating node, and we know the valuations of $2a(p^{m}q+p^{m-1}r+r_{0})+b$ are bigger than or equal to $m-1$, the above implies 
\begin{equation*}
2a(p^{m}q+r_{0})+b=p^{m-1}k_{3},
\end{equation*}
for $k_{3}\in\mathbb{Z}$ depending on $q$. Suppose that $p^{m}\nmid (2a(p^{m}q+r_{0})+b)$. Then this non-terminating node $n=p^{m}q+r_{0}$ would have valuation $\nu_{p}(2a(p^{m}q+r_{0})+b)=2(m-1)$, which contradicts the fact it is non-terminating. Therefore, $p^{m}\mid (2a(p^{m}q+r_{0})+b)$ and this node must have a valuation greater than or equal to $2m$. But furthermore, we must have $p\mid k_{3}$. This and the above equation imply that
\begin{align*}
2a(p^{m}q+p^{m-1}r+r_{0})+b&=2ap^{m}q+2ap^{m-1}r+2ar_{0}+b\\
&=2ap^{m-1}r+p^{m-1}k_{3}\\
&=p^{m-1}(2ar+k_{3}).
\end{align*}
Since $p\mid k_{3}$ and $p\nmid ar$, all other nodes on the $m^{th}$ level must be terminating with valuation $2(m-1)$. This completes the proof.
\end{proof}

\section{Acknowledgments}
The authors would like to thank the National Science Foundation for funding the REU program at Ursinus College through grant 1851948.

\section*{Appendix: examples}
This section includes several examples of the $p$-adic valuation trees for quadratic polynomials.

\subsection*{A tree with one infinite branch}
The figure below shows a $3$-adic valuation tree for the polynomial $f(n)=12n^2+16n+7$. According to case 2 of Theorem\ \ref{theorem:summary}, the tree has one infinite branch.
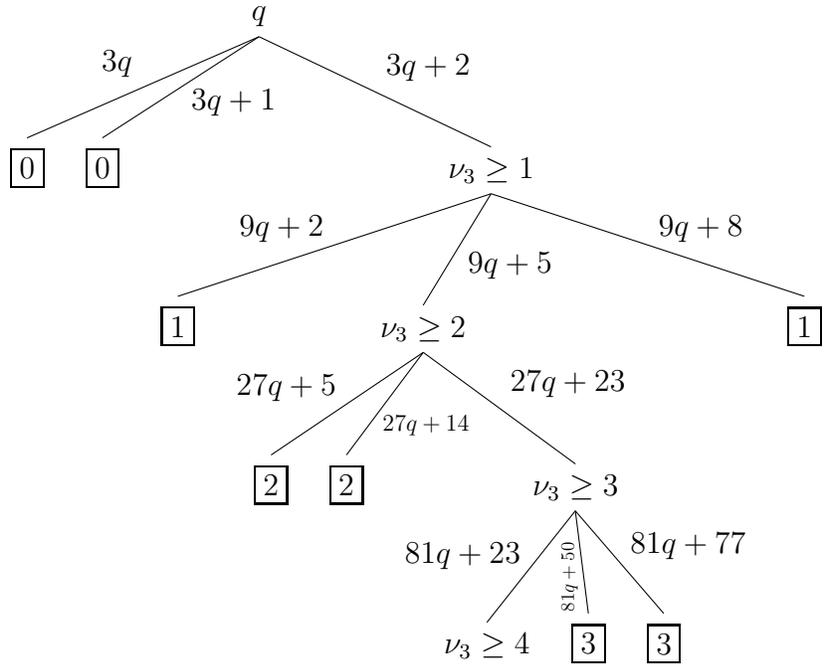
\begin{figure}[H]
\begin{tikzpicture}[level distance=60pt, sibling distance=8pt]
	\Tree [.$q$ 
		\edge node[auto=right]{$3q$}; $\boxed{0}$
		\edge node[auto=right,xshift=40, yshift=-15]{$3q+1$}; $\boxed{0}$
		\edge node[auto=left]{$3q+2$}; 
  	[.$\nu_3\geq 1$
			\edge node[auto=right,yshift=-3]{$9q+2$}; $\boxed{1}$
			\edge node[auto=right, xshift=40, yshift=-15]{$9q+5$}; 
			[.$\nu_3\geq 2$
				\edge node[auto=right,yshift=-3]{$27q+5$}; $\boxed{2}$
				\edge node[auto=right, xshift=35, yshift=-15, scale=0.75]{$27q+14$}; $\boxed{2}$
				\edge node[auto=left]{$27q+23$}; 
                    [.$\nu_3\geq 3$
					\edge node[auto=right, yshift=-5]{$81q+23$}; $\nu_3\geq 4$
					\edge node[auto=right,yshift=10, xshift=-10, scale=0.6,rotate=90]{$81q+50$}; $\boxed{3}$
					\edge node[auto=left,yshift=-3]{$81q+77$}; $\boxed{3}$ ]
    		]
			\edge node[auto=left,yshift=-3]{$9q+8$}; $\boxed{1}$
		]
	]
\end{tikzpicture}

\vspace{0.2in}
  \small
    \begin{tabular}{c|cccc ccccc ccc}
$n$&0&1&2&3&4&5&6&7&8&9&10\\ \hline
$f(n)$&$7$&$35$&87&163&263&387&535&707&903&1123&1367\\
$\nu_{3}(f(n))$&0&0&1&0&0&2&0&0&1&0&0&\\
\end{tabular}

\vspace{0.2in}
\begin{tabular}{c|cccc ccccc ccc}
    $n$&11&12&13&14&15&16&17&18&19\\ \hline
    $f(n)$&1635&1927&2243&2583&2947&3335&3747&4183&4643\\
    $\nu_{3}(f(n))$&1&0&0&2&0&0&1&0&0\\
    \end{tabular}
    \vspace{0.2in}
     \caption{A $3$-adic valuation tree for the polynomial $f(n)=12n^2+16n+7$, followed by the table containing the first twenty terms of the sequence of the $3$-adic valuations generated by $f(n)$.}
    \label{Example:OneInfinteBranch}
\end{figure}

\subsection*{A tree with two infinite branches}
For the next example, consider $f(n)=n^2+1$ and $p=5$. Since $D=-5^0\cdot 4$, the $5$-adic valuation of $D$ is even, and $-4$ is a quadratic residue modulo $5$. By case 3(b) of Theorem\ \ref{theorem:summary}, the valuation tree of $f(n)$ is infinite with two branches. 
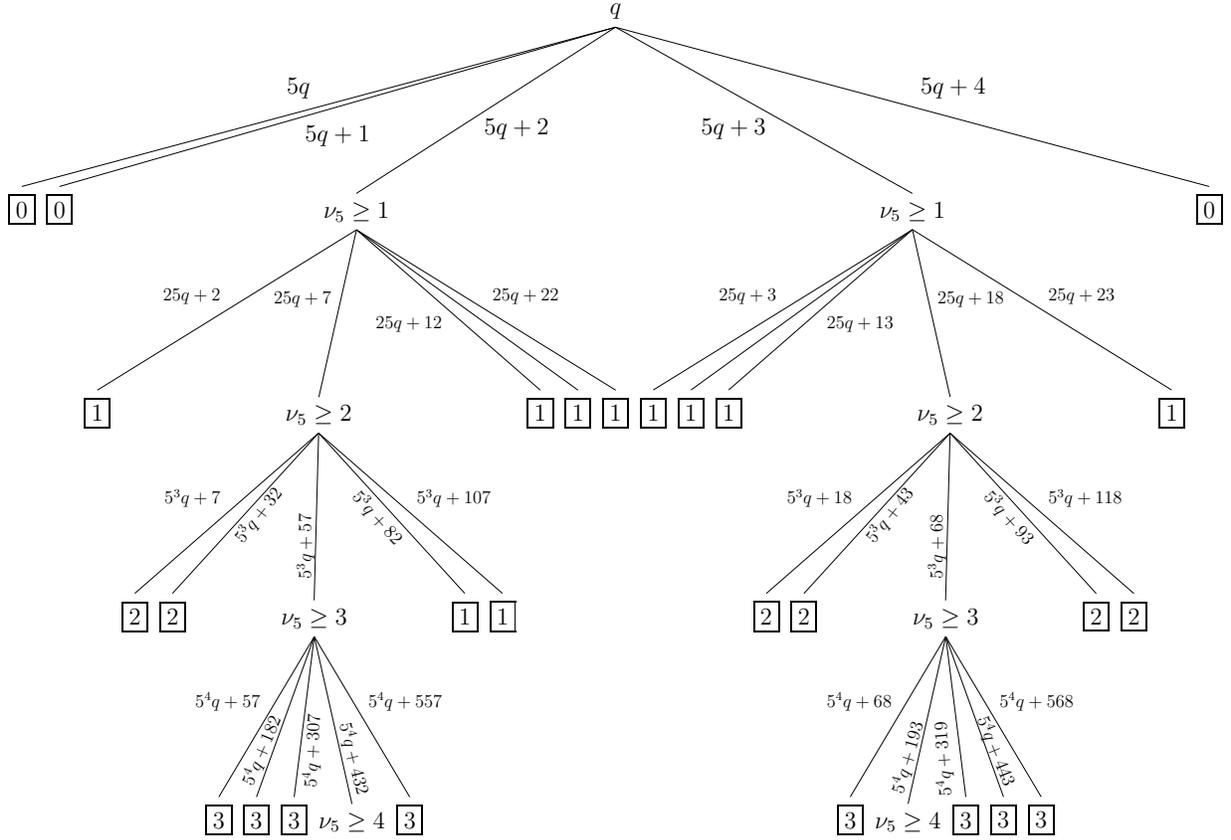
\begin{figure}[H]
\begin{tikzpicture}[scale=.77,level distance=100pt, sibling distance=-2pt]
		\Tree [.$q$ 
			\edge node[auto=right]{$5q$}; $\boxed{0}$
			\edge node[auto=right,below,yshift=-5]{$5q+1$}; $\boxed{0}$
			\edge node[auto=right,below,xshift=15]{$5q+2$}; 
			[.$\nu_5\geq 1$ 
				\edge node[scale=0.75,auto=right]{$25q+2$}; $\boxed{1}$
				\edge node[scale=0.75,auto=right]{$25q+7$};
					[.$\nu_5\geq2$
						\edge node[scale=0.75,auto=right]{$5^3q+7$}; $\boxed{2}$
						\edge node[scale=0.75,auto=left,yshift=-15,xshift=-15,rotate=52]{$5^3q+32$}; $\boxed{2}$
						\edge node[scale=0.75,auto=right,yshift=6,xshift=-7,rotate=90]{$5^3q+57$}; 
							[.$\nu_5\geq3$
								\edge node[scale=0.75,auto=right]{$5^4q+57$}; $\boxed{3}$
								\edge node[scale=0.75,auto=right,xshift=5,rotate=67]{$5^4q+182$}; $\boxed{3}$
								\edge node[scale=0.75,auto=left,yshift=-50,xshift=-6,rotate=87]{$5^4q+307$}; $\boxed{3}$
								\edge node[scale=0.75,auto=left,xshift=-3,rotate=-76]{$5^4q+432$}; $\nu_5\geq 4$
								\edge node[scale=0.75,auto=left]{$5^4q+557$}; $\boxed{3}$
							]
						\edge node[scale=0.75,auto=right,yshift=-15,xshift=15,rotate=-48]{$5^3q+82$}; $\boxed{1}$
						\edge node[scale=0.75,auto=left]{$5^3q+107$}; $\boxed{1}$
					]
				\edge node[scale=0.75,auto=right,rotate=0,yshift=0]{$25q+12$}; $\boxed{1}$ 
				\edge; $\boxed{1}$
				\edge node[scale=0.75,auto=left]{$25q+22$}; $\boxed{1}$
			]
			\edge node[auto=right,below,xshift=-15]{$5q+3$}; 
				[.$\nu_5\geq 1$ 
					\edge node[scale=0.75,auto=right]{$25q+3$}; $\boxed{1}$
					\edge; $\boxed{1}$
					\edge node[scale=0.75,auto=left]{$25q+13$}; $\boxed{1}$
					\edge node[scale=0.75,auto=left]{$25q+18$};
						[.$\nu_5\geq2$
							\edge node[scale=0.75,auto=right]{$5^3q+18$}; $\boxed{2}$
							\edge node[scale=0.75,auto=left,yshift=-15,xshift=-15,rotate=52]{$5^3q+43$}; $\boxed{2}$
							\edge node[scale=0.75,auto=right,yshift=6,xshift=-7,rotate=90]{$5^3q+68$}; 
								[.$\nu_5\geq3$
									\edge node[scale=0.75,auto=right]{$5^4q+68$}; $\boxed{3}$
									\edge node[scale=0.75,auto=right,xshift=5,rotate=74]{$5^4q+193$}; $\nu_5\geq 4$
									\edge node[scale=0.75,auto=right,xshift=-20,rotate=94]{$5^4q+319$}; $\boxed{3}$
									\edge node[scale=0.75,auto=left,xshift=-6,rotate=-68]{$5^4q+443$}; $\boxed{3}$
									\edge node[scale=0.75,auto=left]{$5^4q+568$}; $\boxed{3}$
								]
							\edge node[scale=0.75,auto=right,yshift=-15,xshift=15,rotate=-48]{$5^3q+93$}; $\boxed{2}$
							\edge node[scale=0.75,auto=left]{$5^3q+118$}; $\boxed{2}$
						]
					\edge node[scale=0.75,auto=left]{$25q+23$}; $\boxed{1}$
				]
			\edge node[auto=left]{$5q+4$}; $\boxed{0}$
		]
	\end{tikzpicture}

\vspace{0.2in}
    \small
    \begin{tabular}{c|cccc ccccc ccc}
$n$&0&1&2&3&4&5&6&7&8&9&10\\ \hline
$f(n)$&$1$&$2$&5&10&17&26&37&50&65&82&101\\
$\nu_{5}(f(n))$&0&0&1&1&0&0&0&2&1&0&0&\\
\end{tabular}

\vspace{0.2in}
\begin{tabular}{c|cccc ccccc ccc}
    $n$&11&12&13&14&15&16&17&18&19\\ \hline
    $f(n)$&0&1&0&197&226&257&290&325&362\\
    $\nu_{5}(f(n))$&0&1&1&0&0&0&1&2&0\\
    \end{tabular}
    \vspace{0.2in}
      \caption{A $5$-adic valuation tree for the polynomial $f(n)=n^2+1$, followed by the table containing the first twenty terms of the sequence of the $5$-adic valuations generated by $f(n)$.}
  \label{Example:TwoInfinteBranch}
\end{figure}

\subsection*{A finite tree}
The figures below illustrate two $3$-adic valuation trees from case 3(c) of Theorem\ \ref{theorem:summary}. Consider $f(n)=4n^2+160n-587$. Since $D=160^2+16\cdot 587=3^7\cdot 16$,
it follows that $\nu_3(D)=7$ is odd. By case 3(c) of Theorem\ \ref{theorem:summary}, the tree is finite with 
$\ell=\lceil 7/2\rceil=4$ levels. The second figure is the translated tree $f(n-S_{\ell})$, 
where $S_{\ell}=\frac{b}{2a}=20$. 

\begin{figure}[H]
\begin{tikzpicture}[level distance=60pt, sibling distance=8pt]
	\Tree [.$q$ 
		\edge node[auto=right]{$3q$}; $\boxed{0}$
		\edge node[auto=right,xshift=45, yshift=-10]{$3q+1$};
        [.$\nu_3\geq 2$
			\edge node[auto=right,yshift=-3]{$9q+1$};  $\boxed{2}$
			\edge node[auto=right, xshift=40,yshift=-20]{$9q+4$}; $\boxed{2}$
			\edge node[auto=left,yshift=-3]{$9q+7$};
       [.$\nu_3\geq 4$
				\edge node[auto=right,yshift=-3]{$27q+7$}; 
                 [.$\nu_3\geq 6$
    				\edge node[auto=right,yshift=-10]{$81q+7$}; $\boxed{7}$
    				\edge node[auto=right,scale=0.6,rotate=90,xshift=20,yshift=5]{$81q+34$}; $\boxed{6}$
    				\edge node[auto=left,yshift=-10]{$81q+61$}; $\boxed{6}$
                ]
				\edge node[auto=right,scale=0.75, xshift=5]{$27q+16$}; $\boxed{4}$
				\edge node[auto=left,yshift=-3]{$27q+25$}; $\boxed{4}$
			]
		]
		\edge node[auto=left]{$3q+2$}; $\boxed{0}$
	]
\end{tikzpicture}\\
 \caption{A $3$-adic valuation tree for the polynomial $f(n)=4n^2+160n-587$.}
 \label{Example:FiniteTree}
\end{figure}
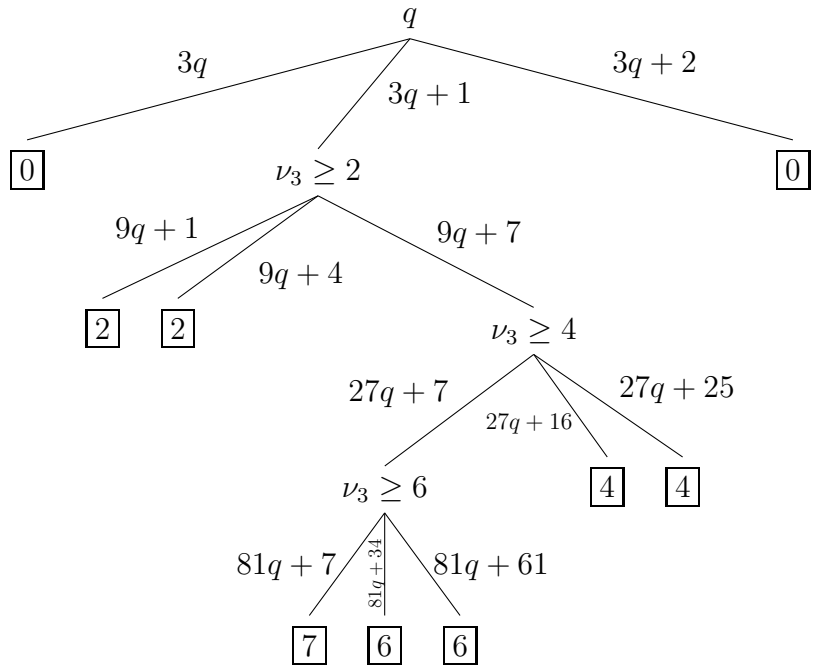

\begin{figure}[H]
\begin{tikzpicture}[level distance=60pt, sibling distance=8pt]
	\Tree [.$q$ 
		\edge node[auto=right]{$3q$}; 
        [.$\nu_3\geq 2$
			\edge node[auto=right]{$9q$}; 
        [.$\nu_3\geq 4$
				\edge node[auto=right,yshift=-3]{$27q$}; 
                [.$\nu_3\geq 6$
    				\edge node[auto=right,yshift=-10]{$81q$}; $\boxed{7}$
    				\edge node[auto=right,scale=0.6,rotate=90,xshift=20,yshift=5]{$81q+27$}; $\boxed{6}$
    				\edge node[auto=left,yshift=-10]{$81q+54$}; $\boxed{6}$
                ]
				\edge node[auto=right,scale=0.75, xshift=3]{$27q+9$}; $\boxed{4}$
				\edge node[auto=left,yshift=-3]{$27q+18$}; $\boxed{4}$                
			]
			\edge node[auto=left, xshift=-40, yshift=-15]{$9q+3$}; $\boxed{2}$
            \edge node[auto=left,yshift=-3]{$9q+6$}; $\boxed{2}$            
		]		
  \edge node[auto=left, xshift=-40, yshift=-15]{$3q+1$}; $\boxed{0}$
        \edge node[auto=left,xshift=3]{$3q+2$};  $\boxed{0}$          
	]
\end{tikzpicture}\\
 \caption{The translated tree: a $3$-adic valuation tree for the polynomial $f(n-20)=4n^2-2187$, used to determine the tree in Figure\ \ref{Example:FiniteTree}.} 
 \label{Example:FiniteTreeTranslation}
\end{figure}
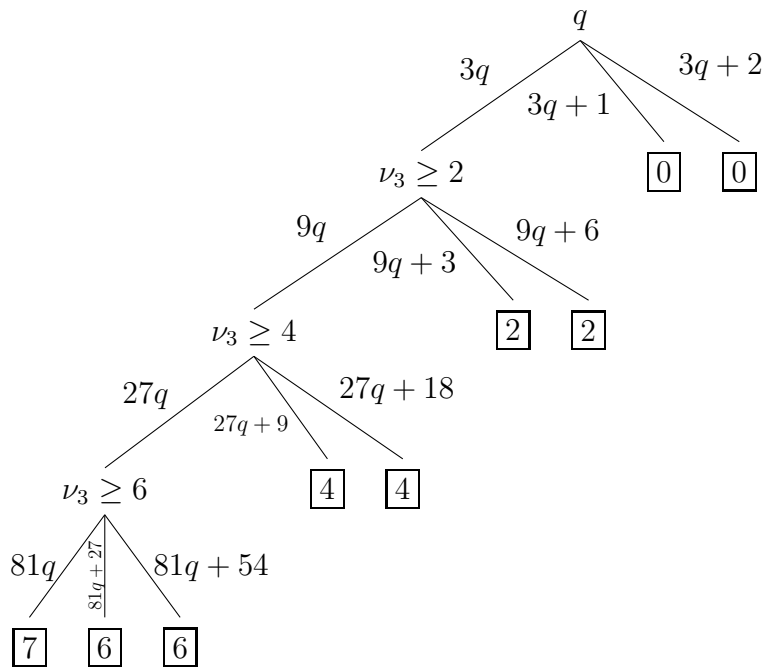

\vspace{0.2in}
 \begin{table}
    \small
    \begin{tabular}{c|cccc ccccc ccc}
$n$&0&1&2&3&4&5&6&7&8&9&10\\ \hline
$f(n)$&$-587$&$-423$&$-251$&$-71$&117&313&517&729&949&1177&1413\\
$\nu_{3}(f(n))$&0&2&0&0&2&0&0&6&0&0&2&\\
\end{tabular}

\vspace{0.2in}
\begin{tabular}{c|cccc ccccc ccc}
    $n$&11&12&13&14&15&16&17&18&19\\ \hline
    $f(n)$&1657&1909&2169&2437&2713&2997&3289&3589&3897\\
    $\nu_{3}(f(n))$&0&0&2&0&0&4&0&0&2\\
    \end{tabular}
    \caption{The first twenty terms of $(\nu_{3}(f(n)))$, generated by $f(n)=4n^2+160n-587$.}
\end{table}

\end{document}